\documentclass[3p,preprint]{elsarticle}
\usepackage[T1]{fontenc}
\usepackage{amsfonts}
\usepackage{amsmath}
\usepackage{geometry}
\usepackage{xcolor}
\usepackage{blindtext}
\usepackage{tikz}
\usepackage{tkz-euclide}
\usepackage{color}
\newtheorem{theorem}{Theorem}[section]

\newtheorem{lemma}[theorem]{Lemma}

\newtheorem{remark}[theorem]{Remark}

\newenvironment{proof}[1][Proof]{\noindent\textbf{#1.} }{\ \rule{0.5em}{0.5em}}

\newcommand{\B}[1]{\boldsymbol{#1}}

\begin{document}

\bigskip

\bigskip
\begin{frontmatter}

\title{A general quadratic enrichment of the Crouzeix–Raviart finite element}
%% Group authors per affiliation:
  \author[address-CS,address-Pau]{Federico Nudo\corref{corrauthor}}
  \cortext[corrauthor]{Corresponding author}
  \ead{federico.nudo@unical.it}

  \address[address-CS]{Department of Mathematics and Computer Science, University of Calabria, Rende (CS), Italy}
  \address[address-Pau]{Laboratoire de Mathématiques et de leurs Applications, UMR CNRS 5142, Université de Pau et des Pays de l'Adour (UPPA), 64000 Pau, France}

\begin{abstract}
The Crouzeix--Raviart finite element method is widely recognized in the field of finite element analysis due to its nonconforming nature. The main goal of this paper is to present a general strategy for enhancing the Crouzeix--Raviart finite element using quadratic polynomial functions and three additional general degrees of freedom. To achieve this, we present a characterization result on the enriched degrees of freedom, enabling to define a new enriched finite element. This general approach is employed to introduce two distinct admissible families of enriched degrees of freedom. Numerical results demonstrate an enhancement in the accuracy of the proposed method when compared to the standard Crouzeix--Raviart finite element, confirming the effectiveness of the proposed enrichment strategy.

\end{abstract}


\end{frontmatter}

\section{Crouzeix–Raviart finite element} \label{sec1}

Let $T\subset\mathbb{R}^2$ be a nondegenerate triangle with vertices $\B{v}_1, \B{v}_2, \B{v}_3$ and barycentric coordinates $\lambda_1, \lambda_2, \lambda_3$. For $j=1,2,3$, we denote by $\Gamma_j$ the edge of $T$ that does not contain the vertex $\B{v}_j$.  The Crouzeix--Raviart finite element is a frequently employed technique in various applications, recognized for its nonconforming nature~\cite{CR1973:2022:CR}. A finite element is considered \textit{nonconforming} when the global approximation obtained is a globally discontinuous function; otherwise, it is referred to as \textit{conforming}.
The use of nonconforming finite elements proves particularly advantageous when dealing with intricate geometries and irregular mesh structures. The Crouzeix--Raviart finite element excels in handling problems featuring solutions with low regularity, making it a valuable choice across a spectrum of engineering and scientific applications~\cite{chatzipantelidis1999finite, hansbo2003discontinuous, burman2005stabilized, carstensen2013discrete, zhu2014analysis, younes2014combination, di2015extension, carstensen2015adaptive, ve2019quasi, chen2019finite, kaltenbach2023error}. Unlike traditional conforming elements, the flexibility of the Crouzeix--Raviart finite element becomes essential, especially in scenarios with discontinuities or challenges posed by irregular meshes.

In order to define the Crouzeix--Raviart finite element some settings are needed. We consider the following linear functionals
\begin{equation*}
 \mathcal{I}_j^{\mathrm{CR}}(f) = \frac{1}{\left\lvert \Gamma_j\right\rvert}\int_{\Gamma_j}f(s)\,ds=\int_{0}^{1}f\left(t\B{v}_{j+1}+ (1-t)\B{v}_{j+2}\right)\, dt, \qquad j=1,2,3,
\end{equation*}
with the convention that
\begin{equation*}
    \B{v}_4 = \B{v}_1, \quad \B{v}_5 = \B{v}_2, \quad \lambda_4 = \lambda_1, \quad \lambda_5 = \lambda_2.
\end{equation*}
We shall assume throughout the paper that the vertex indices are oriented counterclockwise.
The Crouzeix--Raviart finite element is locally defined as
\begin{equation*}
\left(T, \mathbb{P}_{1}(T), \Sigma_T^{\mathrm{CR}}\right),    
\end{equation*}
where 
\begin{equation*}
     \mathbb{P}_{1}(T) = \operatorname{span}\{\lambda_1, \lambda_2, \lambda_3\}, \qquad \Sigma_T^{\mathrm{CR}} = \left\{\mathcal{I}_j^{\mathrm{CR}} \, : \, j=1,2,3\right\}.
\end{equation*}
We denote by
\begin{equation}\label{CRbasis}
    \varphi_i^{\mathrm{CR}}=1-2\lambda_i, \qquad i=1,2,3,
\end{equation}
the basis functions associated to the Crouzeix–Raviart finite element~\cite{CR1973:2022:CR}. This means that the following conditions are satisfied
\begin{equation*}
  \mathbb{P}_1(T)=  \operatorname{span}\{\varphi_i^{\mathrm{CR}}\, :\, i=1,2,3\}
\end{equation*}
and 
\begin{equation*}
    \mathcal{I}_j^{\mathrm{CR}}(\varphi_i^{\mathrm{CR}})=\delta_{ij},
\end{equation*}
where $\delta_{ij}$ is the Kronecker delta symbol. 
In this setting, we define the approximation operator relative to the Crouzeix--Raviart finite element as follows
\begin{equation}\label{pilinch9}
\begin{array}{rcl}
{\Pi}_1^{{\mathrm{CR}}}: C(T) &\rightarrow& \mathbb{P}_1(T)
\\
f &\mapsto& \displaystyle{\sum_{j=1}^{3} \mathcal{I}^{\mathrm{CR}}_j(f)\varphi^{\mathrm{CR}}_j}.
\end{array}
\end{equation}

Then, when employing the Crouzeix--Raviart finite element to solve a differential problem, the domain is partitioned into triangles. For each of these triangles, the solution of the differential problem is approximated using linear polynomials. However, the accuracy of the approximation produced by linear polynomials is relatively low. To address this limitation, a commonly employed strategy is based on the enrichment of finite elements. 
This strategy involves enhancing the representation of the solution within each triangle by adding other functions, also called \textit{enrichment functions}~\cite{Guessab:2022:SAB, Guessab:2016:AADM, Guessab:2016:RM, guessab2017unified, DellAccio:2022:QFA, DellAccio:2022:AUE, DellAccio:2022:ESF, DellAccio2023AGC, DellAccio2023nuovo}. 
Enriched finite elements provide a more refined representation of the solution within each element, contributing to a more accurate global solution.  The main goal of this paper is to present a general strategy to enrich the Crouzeix--Raviart finite element by using quadratic polynomial functions and three general enriched linear functionals.

\section{A comprehensive quadratic enrichment approach}\label{sec2}
The main goal of this work is to present a general enrichment of the Crouzeix--Raviart finite element using three general degrees of freedom and quadratic polynomials.  Since, in general, the higher-order finite element spaces exhibit more local behavior than lower-order spaces.
To achieve this goal, we consider the following enriched linear functionals
\begin{eqnarray*}
\mathcal{F}^{{\mathrm{enr}}}_1: f\in C(T) &\rightarrow& \mathcal{F}^{{\mathrm{enr}}}_1(f)\in \mathbb{R}
\\
\mathcal{F}^{{\mathrm{enr}}}_2: f\in C(T) &\rightarrow& \mathcal{F}^{{\mathrm{enr}}}_2(f)\in \mathbb{R}\\
\mathcal{F}^{{\mathrm{enr}}}_3: f\in C(T) &\rightarrow& \mathcal{F}^{{\mathrm{enr}}}_3(f)\in \mathbb{R}.
\end{eqnarray*}
We consider the triple 
 \begin{equation}\label{trips}
   \mathcal{S}= (T,  \mathbb{P}_2(T),\Sigma_{T}^{{\mathrm{enr}}}),
 \end{equation}
 where 
 \begin{equation*}
     \mathbb{P}_{2}(T)=\operatorname{span}\{\lambda_1,\lambda_2,\lambda_3, \lambda_1^2,\lambda_2^2,\lambda_3^2\}, \qquad \Sigma_{T}^{{\mathrm{enr}}}=  \left\{\mathcal{I}_j^{\mathrm{CR}},\mathcal{F}^{\mathrm{enr}}_{j}\, : \, j=1,2,3\right\}.
 \end{equation*}   
In the following, we establish necessary and sufficient conditions for the enriched linear functionals $\mathcal{F}^{\mathrm{enr}}_{j}$, $j=1,2,3$, to ensure that the enriched triple~\eqref{trips} is a well-defined finite element.  In other words, we aim to prove that the only element $p\in \mathbb{P}_2(T)$ satisfying
\begin{equation*}
\mathcal{I}^{\mathrm{CR}}_j(p)=0, \quad \mathcal{F}^{\mathrm{enr}}_j(p)=0, \qquad j=1,2,3,
\end{equation*}
is $p=0$. For this purpose, we first require some preliminary results. Throughout the paper, we employ the basis functions of $\mathbb{P}_2(T)$, introduced in~\cite{DellAccio:2022:AUE} and given by
\begin{equation}\label{basisAF3}
\mathcal{B}_{AF3}=\{\varphi_i, \phi_i\, : \, i=1,2,3\},
\end{equation}
which satisfy conditions~\cite[Ch. 2]{Davis:1975:IAA}
\begin{equation*}
\mathcal{L}^{\mathrm{enr}}_j(\varphi_i)=\delta_{ij}, \quad  \mathcal{I}_j^{\mathrm{CR}}(\varphi_i)=0, \quad i,j=1,2,3,
\end{equation*}
\begin{equation*}
    \mathcal{L}^{\mathrm{enr}}_j(\phi_i)=0, \quad  \mathcal{I}_j^{\mathrm{CR}}(\phi_i)=\delta_{ij}, \quad i,j=1,2,3,
\end{equation*}
where
\begin{equation}\label{linfunAF3}
\mathcal{L}_j^{\mathrm{enr}}(f)= f(\B{v}_j), \qquad j=1,2,3.
\end{equation}
These functions are known as the basis functions of $\mathbb{P}_2(T)$ associated to the enriched finite element presented in~\cite{DellAccio:2022:AUE}
\begin{equation}\label{AF3}
AF3= \left(T,\mathbb{P}_2(T),{\Sigma}_T^{\mathrm{enr}}\right),
\end{equation}
where
\begin{equation*}
{\Sigma}_T^{{\mathrm{enr}}}= \left\{\mathcal{I}_j^{\mathrm{CR}},\mathcal{L}^{\mathrm{enr}}_j\, : \, j=1,2,3\right\}.
\end{equation*}
In the forthcoming theorem~\cite[Thm. 2.6]{DellAccio:2022:AUE} the explicit expressions of the basis functions are proved.
 \begin{theorem}\label{th2}
        The basis functions $\varphi_i,\phi_i$,  $i=1,2,3,$ of $\mathbb{P}_2(T)$ associated to the finite element $AF3$ have the following expressions
\begin{equation}
\varphi_1 = \lambda_1 (1-  3\lambda_2  -3 \lambda_3), \qquad  \varphi_2=\lambda_2 (1-  3\lambda_1  -3 \lambda_3), \qquad \varphi_3 = \lambda_3 (1-  3\lambda_1  -3 \lambda_2),
    \label{basisphi}
\end{equation}
\begin{equation}
\phi_1 = 6 \lambda_2\lambda_3,\qquad 
   \phi_2 = 6 \lambda_1\lambda_3,\qquad
    \phi_3 = 6\lambda_1\lambda_2.
    \label{basisvarphi}
\end{equation}
      \end{theorem}

      \begin{lemma}\label{lem1}
  Let $p \in \mathbb{P}_2(T)$ such that 
  \begin{equation}\label{condann}
       \mathcal{I}_j^{\mathrm{CR}}(p)=0, \qquad j=1,2,3.
  \end{equation}
Then, we get 
\begin{equation}\label{dnnexprima}
 \begin{bmatrix}
\mathcal{F}_{1}^{\mathrm{enr}}(\varphi_1) &  \mathcal{F}_{1}^{\mathrm{enr}}(\varphi_2) & \mathcal{F}_{1}^{\mathrm{enr}}(\varphi_3)\\
\mathcal{F}_{2}^{\mathrm{enr}}(\varphi_1) & \mathcal{F}_{2}^{\mathrm{enr}}(\varphi_2)& \mathcal{F}_{2}^{\mathrm{enr}}(\varphi_3) \\
\mathcal{F}_{3}^{\mathrm{enr}}(\varphi_1) & \mathcal{F}_{3}^{\mathrm{enr}}(\varphi_2) &  \mathcal{F}_{3}^{\mathrm{enr}}(\varphi_3)
\end{bmatrix}
 \begin{bmatrix}
\mathcal{L}_1^{\mathrm{enr}}(p)\\
\mathcal{L}_2^{\mathrm{enr}}(p)\\
\mathcal{L}_3^{\mathrm{enr}}(p)
\end{bmatrix}
= \begin{bmatrix}
\mathcal{F}_{1}^{{\mathrm{enr}}}(p)\\
\mathcal{F}_{2}^{{\mathrm{enr}}}(p)\\
\mathcal{F}_{3}^{{\mathrm{enr}}}(p)
\end{bmatrix},
\end{equation}
where the linear evaluation functionals $\mathcal{L}^{\mathrm{enr}}_j$, $j=1,2,3$, are defined in~\eqref{linfunAF3}.
\end{lemma}
\begin{proof}
 Let $p \in \mathbb{P}_2(T)$ such that
 \begin{equation*}
\mathcal{I}_j^{\mathrm{CR}}(p)=0, \qquad j=1,2,3.
 \end{equation*}
Then we can express $p$ with respect to the basis~\eqref{basisAF3} as follows
 \begin{equation}\label{eqpb11}
     p=\sum_{k=1}^{3} \mathcal{L}^{\mathrm{enr}}_k(p)\varphi_k+\sum_{k=1}^3 \mathcal{I}^{\mathrm{CR}}_k(p) \phi_k= \sum_{k=1}^{3} \mathcal{L}^{\mathrm{enr}}_k(p)\varphi_k.
 \end{equation}
Applying the linear functional $\mathcal{F}_{j}^{\mathrm{enr}}$, $j=1,2,3$, to both sides of~\eqref{eqpb11}, we obtain the following equations
\begin{eqnarray*}
    \mathcal{F}_{1}^{\mathrm{enr}}(p) &=& \sum_{k=1}^{3} \mathcal{L}_k^{\mathrm{enr}}(p)\mathcal{F}_{1}^{\mathrm{enr}}(\varphi_k)\\
        \mathcal{F}_{2}^{\mathrm{enr}}(p) &=& \sum_{k=1}^{3} \mathcal{L}_k^{\mathrm{enr}}(p)\mathcal{F}_{2}^{\mathrm{enr}}(\varphi_k)\\
        \mathcal{F}_{3}^{\mathrm{enr}}(p) &=& \sum_{k=1}^{3} \mathcal{L}_k^{\mathrm{enr}}(p)\mathcal{F}_{3}^{\mathrm{enr}}(\varphi_k).
\end{eqnarray*}
These equations can be expressed in the matrix form as follows
\begin{equation}\label{dnnexprdima}
 \begin{bmatrix}
\mathcal{F}_{1}^{\mathrm{enr}}(\varphi_1) &  \mathcal{F}_{1}^{\mathrm{enr}}(\varphi_2) & \mathcal{F}_{1}^{\mathrm{enr}}(\varphi_3)\\
\mathcal{F}_{2}^{\mathrm{enr}}(\varphi_1) & \mathcal{F}_{2}^{\mathrm{enr}}(\varphi_2)& \mathcal{F}_{2}^{\mathrm{enr}}(\varphi_3) \\
\mathcal{F}_{3}^{\mathrm{enr}}(\varphi_1) & \mathcal{F}_{3}^{\mathrm{enr}}(\varphi_2) &  \mathcal{F}_{3}^{\mathrm{enr}}(\varphi_3)
\end{bmatrix}
 \begin{bmatrix}
\mathcal{L}_1^{\mathrm{enr}}(p)\\
\mathcal{L}_2^{\mathrm{enr}}(p)\\
\mathcal{L}_3^{\mathrm{enr}}(p)
\end{bmatrix}
= \begin{bmatrix}
\mathcal{F}_{1}^{{\mathrm{enr}}}(p)\\
\mathcal{F}_{2}^{{\mathrm{enr}}}(p)\\
\mathcal{F}_{3}^{{\mathrm{enr}}}(p)
\end{bmatrix}.
\end{equation}
\end{proof}

In the following, we denote by
\begin{equation}\label{matrixN}
    N= \begin{bmatrix}
\mathcal{F}_{1}^{\mathrm{enr}}(\varphi_1) &  \mathcal{F}_{1}^{\mathrm{enr}}(\varphi_2) & \mathcal{F}_{1}^{\mathrm{enr}}(\varphi_3)\\
\mathcal{F}_{2}^{\mathrm{enr}}(\varphi_1) & \mathcal{F}_{2}^{\mathrm{enr}}(\varphi_2)& \mathcal{F}_{2}^{\mathrm{enr}}(\varphi_3) \\
\mathcal{F}_{3}^{\mathrm{enr}}(\varphi_1) & \mathcal{F}_{3}^{\mathrm{enr}}(\varphi_2) &  \mathcal{F}_{3}^{\mathrm{enr}}(\varphi_3)
\end{bmatrix}.
\end{equation}
\begin{remark}\label{remarkimp}
The previous theorem establishes a relationship between the evaluations of a polynomial $p$ at the vertices of the triangle $T$ and the values of the enrichment functionals $\mathcal{F}_j^{\mathrm{enr}}$, $j=1,2,3,$ on the same polynomial, under the condition that $p$ satisfies~\eqref{condann}. In other words, if the matrix $N$ is nonsingular, a consequence of the previous theorem is that for any polynomial $p$ satisfying~\eqref{condann}, we have
\begin{equation}\label{condsupimp}
 \begin{bmatrix}
\mathcal{L}_1^{\mathrm{enr}}(p)\\
\mathcal{L}_2^{\mathrm{enr}}(p)\\
\mathcal{L}_3^{\mathrm{enr}}(p)
\end{bmatrix}=
\begin{bmatrix}
0\\
0\\
0
\end{bmatrix}\iff \begin{bmatrix}
\mathcal{F}_1^{\mathrm{enr}}(p)\\
\mathcal{F}_2^{\mathrm{enr}}(p)\\
\mathcal{F}_3^{\mathrm{enr}}(p)
\end{bmatrix}=
\begin{bmatrix}
0\\
0\\
0
\end{bmatrix}.
\end{equation}
\end{remark}

Now, we give a necessary and sufficient condition on the enriched linear functionals $\mathcal{F}_j^{\mathrm{enr}}$, $j=1,2,3,$ under which the triple $\mathcal{S}$ defined in~\eqref{trips} is a finite element. 
\begin{theorem}\label{th1} 
The triple $\mathcal{S}$, defined in~\eqref{trips}, is a finite element if and only if the matrix $N$, as defined in~\eqref{matrixN}, is nonsingular.
\end{theorem}
\begin{proof}
Firstly, let us assume that the matrix $N$, defined in~\eqref{matrixN}, is nonsingular, and we aim to establish that $\mathcal{S}$ is a finite element. For this purpose, we consider $p\in \mathbb{P}_2(T)$ such that 
\begin{equation*}
  \mathcal{I}_j^{\mathrm{CR}}(p)= 0, \quad
  \mathcal{F}_{j}^{\mathrm{enr}}(p)=0, \qquad j=1,2,3.
\end{equation*}
To establish $\mathcal{S}$ as a finite element, we must prove that $p=0$. By~\eqref{condsupimp} of Remark~\ref{remarkimp}, we deduce that the polynomial $p$ also satisfies
 \begin{equation*}
  \mathcal{I}_j^{\mathrm{CR}}(p)= 0, \quad
 \mathcal{L}^{\mathrm{enr}}_j(p)=0, \qquad j=1,2,3. 
\end{equation*} 
Then, since the triple $AF3$, defined in~\eqref{AF3}, is a finite element, as established in~\cite[Theorem 2.6]{DellAccio:2022:AUE}, we conclude that $p=0$. 

For the reverse implication, let us assume that the matrix $N$ is singular. Then, there exists 
\begin{equation}\label{alpa}
\B \alpha=[\alpha_1,\alpha_2,\alpha_3]^T\neq[0,0,0]^T    
\end{equation}
 such that
\begin{equation}\label{dnnexprdima1}
    N\B \alpha=\begin{bmatrix}
0\\
0\\
0
\end{bmatrix}.
\end{equation}
Since $AF3$ is a finite element, we can find a polynomial $p\in \mathbb{P}_2(T)$, different from zero, such that
\begin{equation*}
    \mathcal{L}_j^{\mathrm{enr}}(p)=\alpha_j, \quad \mathcal{I}^{\mathrm{CR}}_j(p)=0, \qquad j=1,2,3.
\end{equation*}
 Then, by~\eqref{dnnexprdima} and~\eqref{dnnexprdima1}, we have found a polynomial $p\neq 0$ such that
\begin{equation*}
    \mathcal{I}_j^{\mathrm{CR}}(p)=0, \quad \mathcal{F}_j^{\mathrm{enr}}(p)=0, \qquad j=1,2,3,
\end{equation*}
contradicting our assumption that $\mathcal{S}$ is a finite element.
\end{proof}

In the following, we assume that the matrix $N$ defined in~\eqref{matrixN} is nonsingular, and we denote its inverse by
\begin{equation}\label{invN}
       N^{-1}= [\B c_1,\B c_2, \B c_3],
\end{equation}
where $\B c_i$, $i=1,2,3,$ is a column vector. An immediate implication of Theorem~\ref{th1} is the existence of a basis
\begin{equation*}
    \mathcal{B}_{\mathcal{S}}=\{\rho_i, \tau_i\, :\, i=1,2,3\}
\end{equation*}
of $\mathbb{P}_2(T)$ satisfying the following properties
\begin{eqnarray}
\label{propvarphi}
&& \mathcal{I}_j^{\mathrm{CR}}(\rho_i) = \delta_{ij}, \quad \mathcal{F}^{\mathrm{enr}}_{j}(\rho_i) = 0, \qquad i,j=1,2,3, \\ \label{propphi}
&&\mathcal{I}_j^{\mathrm{CR}}(\tau_i) = 0, \quad \mathcal{F}^{\mathrm{enr}}_{j}(\tau_i) = \delta_{ij}, \qquad i,j=1,2,3.
\end{eqnarray}
These functions, referred to as the basis functions of $\mathbb{P}_2(T)$ associated to the enriched finite element $\mathcal{S}$, are explicitly expressed in the subsequent theorem. In the following we denote by  $\left\langle \cdot ,\cdot\right\rangle$ the standard inner product in $\mathbb{R}^3$.

\begin{theorem}\label{th2allalf} 
We assume that the matrix $N$, defined in~\eqref{matrixN} is a nonsingular matrix whose inverse can be written as in~\eqref{invN}. Then, the basis functions $\rho_i,\tau_i$,  $i=1,2,3,$ of $\mathbb{P}_2(T)$ associated to the finite element $\mathcal{S}$ have the following expressions
\begin{equation}\label{rhoetau}
    \rho_i=\left\langle \B w_i, \B \varphi \right\rangle +\phi_i, \qquad \tau_i=\left\langle \B c_i, \B \varphi \right\rangle,
\end{equation}
where 
\begin{equation}\label{notimpsup}
  \B \varphi=\begin{bmatrix}
\varphi_{1}\\
\varphi_{2}\\
\varphi_{3}
\end{bmatrix}, \qquad   \B w_i=-\B c_1 \mathcal{F}_1^{\mathrm{enr}}(\phi_i)-\B c_2\mathcal{F}_2^{\mathrm{enr}}(\phi_i)-\B c_3\mathcal{F}_3^{\mathrm{enr}}(\phi_i), \qquad i=1,2,3
\end{equation}
and $\varphi_i,\psi_i$, $i=1,2,3,$ are defined in~\eqref{basisphi} and~\eqref{basisvarphi}. 
\end{theorem}
\begin{proof}
Without loss of generality, we prove the theorem for $i=1$.

Firstly, we prove the expression for $\rho_1$. 
Since $\rho_1\in\mathbb{P}_2(T)$, it can be expressed with respect to the basis~\eqref{basisAF3} associated to the finite element $AF3$ as follows
\begin{equation*}
    \rho_1=\sum_{k=1}^3 \mathcal{L}^{\mathrm{enr}}_k(\rho_1) \varphi_k+\sum_{k=1}^3 \mathcal{I}^{\mathrm{CR}}_k(\rho_1) \phi_k.
\end{equation*}
By~\eqref{propvarphi}, we have
\begin{equation} \label{djeie}
    \rho_1=\sum_{k=1}^3 \mathcal{L}^{\mathrm{enr}}_k(\rho_1) \varphi_k+\phi_1=\left\langle \B{L}(\rho_1), \B \varphi\right\rangle +\phi_1,
\end{equation}
where
\begin{equation*}
\B{L}(\rho_1)=\begin{bmatrix}
\mathcal{L}_{1}^{{\mathrm{enr}}}(\rho_1)\\
\mathcal{L}_{2}^{{\mathrm{enr}}}(\rho_1)\\
\mathcal{L}_{3}^{{\mathrm{enr}}}(\rho_1)
\end{bmatrix}. 
\end{equation*}
Now, let us compute the vector $\B L(\rho_1)$. 
To this aim, by applying the linear functionals $\mathcal{F}^{\mathrm{enr}}_{j}$, $j=1,2,3$, to both sides of~\eqref{djeie}, leveraging~\eqref{propvarphi}, we obtain the following linear system
\begin{equation*}
\begin{bmatrix}
0\\
0\\
0
\end{bmatrix}=N
 \B{L}(\rho_1)
+\begin{bmatrix}
\mathcal{F}_{1}^{{\mathrm{enr}}}(\phi_1)\\
\mathcal{F}_{2}^{{\mathrm{enr}}}(\phi_1)\\
\mathcal{F}_{3}^{{\mathrm{enr}}}(\phi_1)
\end{bmatrix},
\end{equation*}
where $N$ is defined in~\eqref{matrixN}. 
Thus, since we have assumed that $N$ is a nonsingular matrix we get 
\begin{equation}\label{mdiendekm}
\B{L}(\rho_1)=-N^{-1}
 \begin{bmatrix}
\mathcal{F}_1^{\mathrm{enr}}(\phi_1)\\
\mathcal{F}_2^{\mathrm{enr}}(\phi_1)\\
\mathcal{F}_3^{\mathrm{enr}}(\phi_1)
\end{bmatrix}.
\end{equation}
By~\eqref{invN}, we have 
\begin{equation*}
  \B{L}(\rho_1)=-\B c_1 \mathcal{F}_1^{\mathrm{enr}}(\phi_1)-\B c_2\mathcal{F}_2^{\mathrm{enr}}(\phi_1)-\B c_3\mathcal{F}_3^{\mathrm{enr}}(\phi_1).
\end{equation*}
Substituting this value in~\eqref{djeie}, the desired statement follows.

It remains to prove the expression for $\tau_1$.  Since $\tau_1\in\mathbb{P}_2(T)$, it can be expressed with respect to the basis~\eqref{basisAF3} associated to the finite element $AF3$ as follows
\begin{equation*}
    \tau_1=\sum_{k=1}^3 \mathcal{L}^{\mathrm{enr}}_k(\tau_1) \varphi_k+\sum_{k=1}^3 \mathcal{I}^{\mathrm{CR}}_k(\tau_1) \phi_k.
\end{equation*}
By~\eqref{propphi}, we have
\begin{equation} \label{djeie1}
    \tau_1=\sum_{k=1}^3 \mathcal{L}_k(\tau_1) \varphi_k=\left\langle \B{L}(\tau_1),\B \varphi\right\rangle,
\end{equation}
where
\begin{equation*}
\B{L}(\tau_1)= \begin{bmatrix}
\mathcal{L}^{\mathrm{enr}}_1(\tau_1)\\
\mathcal{L}^{\mathrm{enr}}_2(\tau_1)\\
\mathcal{L}^{\mathrm{enr}}_3(\tau_1)
\end{bmatrix}.
\end{equation*}
Now, let us compute the vector $\B{L}(\tau_1)$. 
To this aim, by applying the linear functionals $\mathcal{F}^{\mathrm{enr}}_{j}$, $j=1,2,3$, to both sides of~\eqref{djeie1}, leveraging~\eqref{propphi}, we obtain the following linear system
\begin{equation*}
\begin{bmatrix}
1\\
0\\
0
\end{bmatrix}=N
\B{L}(\tau_1)
\end{equation*}
where $N$ is defined in~\eqref{matrixN}. Then, we get
\begin{equation*}
    N^{-1}\begin{bmatrix}
1\\
0\\
0
\end{bmatrix}= \B{L}(\tau_1).
\end{equation*}
By~\eqref{invN}, we have 
\begin{equation*}
   \B{L}(\tau_1)= \B c_1.
\end{equation*}
Substituting this value in~\eqref{djeie1}, the desired expression for $\tau_1$ is proved. 
    Analogously, the theorem can be established for $i=2$ and $i=3$; consequently, the thesis follows.
\end{proof}

\begin{theorem}
 We assume that the matrix $N$, defined in~\eqref{matrixN} is a nonsingular matrix whose inverse can be written as in~\eqref{invN}. The approximation operator relative to the enriched finite element $\mathcal{S}$
\begin{equation}
\begin{array}{rcl}
{\Pi}_{\mathcal{S}}^{{\mathrm{enr}}}: C(T) &\rightarrow& \mathbb{P}_2(T)
\\
f &\mapsto& \displaystyle{\sum_{j=1}^{3}  \mathcal{I}^{\mathrm{CR}}_j(f)\rho_j+ \sum_{j=1}^{3}\mathcal{F}^{\mathrm{enr}}_{j}(f)}\tau_j,
\end{array}
\label{pilinch9C}
\end{equation}
reproduces all polynomials of $\mathbb{P}_2(T)$ and satisfies 
\begin{eqnarray*}
\mathcal{I}^{\mathrm{CR}}_j\left({\Pi}_{\mathcal{S}}^{\mathrm{enr}}[f]\right)&=&\mathcal{I}^{\mathrm{CR}}_j(f), \qquad j=1,2,3, \\ \mathcal{F}^{\mathrm{enr}}_{j}\left({\Pi}_{\mathcal{S}}^{\mathrm{enr}}[f]\right)&=&\mathcal{F}^{\mathrm{enr}}_{j}(f), \qquad j=1,2,3.
\end{eqnarray*} 
\end{theorem}
\begin{proof}
The proof is a consequence of~\eqref{propvarphi} and~\eqref{propphi}.
\end{proof}

\begin{remark}
    If we consider the following enriched linear functionals 
    \begin{equation*}
\mathcal{F}^{\mathrm{enr}}_j(f)=f(\B v_j), \qquad j=1,2,3,
    \end{equation*}
then, the triple $\mathcal{S}$ corresponds to the enriched finite element $AF3$ introduced in~\cite{DellAccio:2022:AUE}.
\end{remark}

\section{Admissible enriched linear functionals}\label{admenrlinfun}
In this section, we present two families of \textit{admissible enriched linear functionals} $\mathcal{F}_j^{\mathrm{enr}}$, $j=1,2,3,$ that is, those which satisfy
    \begin{equation*}
           \operatorname{det}(N)\neq0,
    \end{equation*}
where $N$ is defined in~\eqref{matrixN}.

\subsection{Admissible enriched linear functionals of first class}
We denote by
\begin{equation*}
    \B m_{1}=\frac{\B v_{2}+\B v_{3}}{2}\in \Gamma_1, \qquad  \B m_{2}=\frac{\B v_{3}+\B v_{1}}{2}\in \Gamma_2, \qquad  \B m_{3}=\frac{\B v_{1}+\B v_{2}}{2}\in \Gamma_3, 
\end{equation*}
the midpoints of the triangle $T$.
For a fixed parameter $\gamma>-1$, we consider the following enriched linear functionals 
\begin{equation}\label{enrfun1}
\mathcal{F}_{j,\gamma}^{\mathrm{enr}}(f)=\int_{0}^1 w_{\gamma}(t)f(t \B m_{j+1}+(1-t)\B m_{j+2})\, dt, \qquad j=1,2,3,
\end{equation}
where $w_{\gamma}$ is the weight function on $[0,1]$, defined by
\begin{equation*}
    w_{\gamma}(t)=t^{\gamma}(1-t)^{\gamma}.
\end{equation*}
Here, we use the convention that
\begin{equation*}
    \B m_4=\B m_1, \qquad \B m_5=\B m_2.
\end{equation*}
Below, we demonstrate the admissibility of the enriched linear functionals given by~\eqref{enrfun1}. To this aim, we denote by
\begin{equation}\label{beta}
    B(z_1,z_2)=\int_{0}^{1}u^{z_1-1}(1-u)^{z_2-1} du, \qquad z_1,z_2>-1,
\end{equation}
the classical Euler beta function~\cite{Abramowitz:1948:HOM}. This function is symmetric, i.e. 
\begin{equation}\label{simmbeta}
    B(z_1,z_2)=B(z_2,z_1), \qquad z_1,z_2>-1,
\end{equation}
and satisfies the following property
\begin{equation}\label{propbetafun}
B(z_1+1,z_2)=\frac{z_1}{z_1+z_2}B(z_1,z_2),  \qquad z_1,z_2>-1.  
\end{equation}
Moreover, we denote also by
\begin{equation}\label{impnot}
\sigma_{\gamma}=B(\gamma+1,\gamma+1), \qquad K_{\gamma}=-\frac{(5\gamma+6)}{8(2\gamma+3)}.
\end{equation}
We recall the following property of the barycenter coordinates
\begin{equation}\label{prop2}
\lambda_i(t \B{x} + (1-t) \B{y}) = t \lambda_i(\B{x}) + (1-t)\lambda_i(\B{y}), \qquad \B{x}, \B{y}\in T, \quad i=1,2,3,
\end{equation}
and
\begin{equation}\label{prop3}
    \lambda_i(\B m_j)=\frac{1}{2}(1-\delta_{ij}), \qquad i,j=1,2,3.
\end{equation}
\begin{theorem}\label{t1d}
For any $\gamma\neq 0$, the enriched linear functionals~\eqref{enrfun1} are admissible. 
\end{theorem}
\begin{proof}
To prove this theorem, we must prove that the matrix $N$, defined in~\eqref{matrixN}, is nonsingular. 
Using~\eqref{basisphi}, we conduct direct calculations, leveraging the properties of the beta function~\eqref{simmbeta} and~\eqref{propbetafun}, as well as the properties of the barycentric coordinates~\eqref{prop2} and~\eqref{prop3}, we can derive
\begin{equation*}
\mathcal{F}_{j,\gamma}^{\mathrm{enr}}(\varphi_j)=-\frac{\sigma_{\gamma}}{4}, \qquad \mathcal{F}_{j,\gamma}^{\mathrm{enr}}(\varphi_i)=\sigma_{\gamma}K_{\gamma}, \qquad i,j=1,2,3, \quad i\neq j,
\end{equation*}
where $\sigma_{\gamma},K_{\gamma}$ are defined in~\eqref{impnot}.
Then, the matrix~\eqref{matrixN} can be written as
\begin{equation}\label{matrixN1}
    N= \sigma_{\gamma}\begin{bmatrix}
-\frac{1}{4} & K_{\gamma} &  K_{\gamma} \\
K_{\gamma} & -\frac{1}{4} & K_{\gamma} \\
K_{\gamma} &  K_{\gamma} &-\frac{1}{4}
\end{bmatrix}.
\end{equation}
Its determinant is given by
\begin{equation*}
\frac{\gamma^2(7\gamma+9)}{256(2 \gamma+3)^3}\sigma_{\gamma}^3,
\end{equation*}
which is different from zero for any $\gamma>-1$ and $\gamma\neq0$.
\end{proof}

In the next theorem, we compute the explicit expression of the basis function of $\mathbb{P}_2(T)$ associated to the enriched finite element
\begin{equation*}
    GN_{\gamma}=(T,\mathbb{P}_2(T),\Sigma_{\gamma, T}^{\mathrm{enr}})
\end{equation*}
where
\begin{equation*}
\Sigma_{\gamma,T}^{\mathrm{enr}}=\left\{\mathcal{I}_j^{\mathrm{CR}}, \mathcal{F}_{j,\gamma}^{\mathrm{enr}}\, :\, j=1,2,3\right\}
\end{equation*}
and $\mathcal{F}_{j,\gamma}^{\mathrm{enr}}$, $j=1,2,3,$ is defined in~\eqref{enrfun1}. 

\begin{theorem}\label{th2allal1f} 
The basis functions $\rho_i,\tau_i$,  $i=1,2,3,$ of $\mathbb{P}_2(T)$ associated to the finite element $GN_{\gamma}$ have the following expressions
\begin{equation}
    \rho_i=c_{\gamma} \varphi_i +d_{\gamma}\sum_{\substack{k=1 \\ k \neq i}}^{3} \varphi_k + \phi_i,  \quad
    \tau_i=\frac{1}{\sigma_{\gamma}\Delta_{\gamma}}\left((-4K_{\gamma}+1)\varphi_i+ 4K_{\gamma}\sum_{\substack{k=1 \\ k \neq i}}^{3} \varphi_k\right), \qquad i=1,2,3,
\end{equation}
where
\begin{equation}\label{setting}
    c_{\gamma}=\frac{3(11 \gamma^2+ 33 \gamma+24)}{\gamma(7\gamma+9)}, \qquad d_{\gamma}=-\frac{3(\gamma+3) (3\gamma+4)}{\gamma(7\gamma+9)}, \qquad \Delta_{\gamma}= \frac{\gamma(7\gamma+9)}{8(2\gamma+3)^2}.
\end{equation}
      \end{theorem}
\begin{proof}
To prove this theorem, we use the general Theorem~\ref{th2allalf}. To this aim, by~\eqref{matrixN1} and by straightforward calculations, we get
\begin{equation}\label{matfrixN1}
    N^{-1}= \frac{1}{\sigma_{\gamma}\Delta_{\gamma}}\begin{bmatrix}
-4K_{\gamma}+1 & 4K_{\gamma} & 4K_{\gamma}\\
4K_{\gamma} & -4K_{\gamma}+1 & 4K_{\gamma} \\
4K_{\gamma} &  4K_{\gamma} &-4K_{\gamma}+1
\end{bmatrix}.
\end{equation}
Moreover, by using the
properties of the beta function~\eqref{simmbeta} and~\eqref{propbetafun} as well as the properties of the barycentric coordinates~\eqref{prop2} and~\eqref{prop3}, we can derive
\begin{equation*}
\mathcal{F}_{j,\gamma}^{\mathrm{enr}}(\phi_j)=\frac{3(\gamma+1)}{4(2\gamma+3)}\sigma_{\gamma}, \qquad \mathcal{F}_{j,\gamma}^{\mathrm{enr}}(\phi_i)=\frac{3}{4}\sigma_{\gamma}, \qquad i,j=1,2,3, \quad i\neq j.
\end{equation*}
Then, by~\eqref{rhoetau} and~\eqref{notimpsup}, the statement follows.  
\end{proof}

\subsection{Admissible enriched linear functionals of second class}
In this section,  we provide another example of admissible enriched linear functionals. To achieve this, let us denote by $\B m^{\star}$ the barycenter of $T$, that is
\begin{equation*}
    \B m^{\star}=\frac{\B v_1+\B v_2+\B v_3}{3}.
\end{equation*} 
In what follows, we adopt the same notations as in the previous subsection. For a fixed parameter $\mu>-1$, we consider the following enriched linear functionals 
\begin{equation}\label{enrfjun1}
\mathcal{G}_{j,\mu}^{\mathrm{enr}}(f)=\int_{0}^1 w_{\mu}(t)f(t \B m_{j}+(1-t)\B m^{\star})\, dt, \qquad j=1,2,3.
\end{equation}
Additionally, we define
\begin{equation}\label{fjfoidf}
D_{\mu}=-\frac{3\mu+4}{3(2\mu+3)}, \qquad H_{\mu}=-\frac{15\mu+22}{12(2\mu+3)}.
\end{equation}
The next theorem establishes the admissibility of the enriched linear functionals provided in~\eqref{enrfjun1}. To this end, we invoke the following property of the barycentric coordinates
\begin{equation}\label{tosub}
\lambda_j(\B m^{\star})=\frac{1}{3}, \qquad j=1,2,3. 
\end{equation}
\begin{theorem}
The enriched linear functionals~\eqref{enrfjun1} are admissible. 
\end{theorem}
\begin{proof}
As in Theorem~\ref{t1d}, we must prove that the matrix $N$, defined in~\eqref{matrixN}, is nonsingular. 
Using~\eqref{basisphi}, we conduct direct calculations, leveraging the properties of the beta function~\eqref{simmbeta} and~\eqref{propbetafun}, as well as the properties of the barycentric coordinates~\eqref{prop2},~\eqref{prop3} and~\eqref{tosub}, we can derive
\begin{equation*}
\mathcal{G}_{j,\mu}^{\mathrm{enr}}(\varphi_j)=\frac{\sigma_{\mu}}{2}D_{\mu}, \qquad \mathcal{G}_{j,\gamma}^{\mathrm{enr}}(\varphi_i)=\frac{\sigma_{\mu}}{2}H_{\mu}, \qquad i,j=1,2,3, \quad i\neq j,
\end{equation*}
where $\sigma_{\mu}$ is defined in~\eqref{impnot} and $H_{\mu}, D_{\mu}$ are defined in~\eqref{fjfoidf}.
Then, the matrix~\eqref{matrixN} can be written as
\begin{equation}\label{matrixN1pgb}
    N= \frac{\sigma_{\mu}}{2}\begin{bmatrix}
D_{\mu} & H_{\mu} &   H_{\mu} \\
 H_{\mu} &  D_{\mu} &  H_{\mu} \\
 H_{\mu} &  H_{\mu} & D_{\mu}
\end{bmatrix}.
\end{equation}
Its determinant is given by
\begin{equation*}
-\frac{(\mu+2)^2(7\mu+10)}{256(2\mu+3)^3}\sigma_{\mu}^3
\end{equation*}
which is different from zero for any $\mu>-1$.
\end{proof}

In the next theorem, we compute the explicit expression of the basis function of $\mathbb{P}_2(T)$ associated to the enriched finite element
\begin{equation*}
    PN_{\mu}=(T,\mathbb{P}_2(T),\Sigma_{\mu, T}^{\mathrm{enr}})
\end{equation*}
where
\begin{equation*}
\Sigma_{\mu,T}^{\mathrm{enr}}=\left\{\mathcal{I}_j^{\mathrm{CR}}, \mathcal{G}_{j,\mu}^{\mathrm{enr}}\, :\, j=1,2,3\right\}
\end{equation*}
and $\mathcal{G}_{j,\mu}^{\mathrm{enr}}$, $j=1,2,3,$ is defined in~\eqref{enrfjun1}. 

\begin{theorem}\label{th2allgal1f} 
The basis functions $\rho_i,\tau_i$,  $i=1,2,3,$ of $\mathbb{P}_2(T)$ associated to the finite element $PN_{\mu}$ have the following expressions
\begin{equation}
    \rho_i=r_{\mu} \varphi_i +q_{\mu}\sum_{\substack{k=1 \\ k \neq i}}^{3} \varphi_k + \phi_i,  \quad
    \tau_i=\frac{2}{\sigma_{\mu}\Omega_{\mu}}\left((D_{\mu}+H_{\mu})\varphi_i- H_{\mu}\sum_{\substack{k=1 \\ k \neq i}}^{3} \varphi_k\right), \qquad i=1,2,3,
\end{equation}
where
\begin{equation}\label{settingnewfamily}
    r_{\mu}=-\frac{125\mu^2+372\mu+276}{3(\mu+2)(7\mu+10)}, \qquad q_{\mu}=\frac{85\mu^2+264\mu+204}{3(\mu+2)(7\mu+10)}, \qquad \Omega_{\mu}=-\frac{(\mu+2)(7\mu+10)}{8(2\mu+3)^2}.
\end{equation}
      \end{theorem}
\begin{proof}
To prove this theorem, we use the general Theorem~\ref{th2allalf}. To this aim, by~\eqref{matrixN1pgb} and by straightforward calculations, we get
\begin{equation}\label{matfrixN1g}
    N^{-1}= \frac{2}{\sigma_{\mu}\Omega_{\mu}}\begin{bmatrix}
D_{\mu}+H_{\mu} & -H_{\mu} & -H_{\mu}\\
-H_{\mu} & D_{\mu}+H_{\mu} & -H_{\mu} \\
-H_{\mu} & -H_{\mu} &D_{\mu}+H_{\mu}
\end{bmatrix}.
\end{equation}
Moreover, by using the
properties of the beta function~\eqref{simmbeta} and~\eqref{propbetafun} as well as the properties of the barycentric coordinates~\eqref{prop2},~\eqref{prop3} and~\eqref{tosub}, we can derive 
\begin{equation*}
\mathcal{G}_{j,\mu}^{\mathrm{enr}}(\phi_j)=\frac{25\mu+38}{12(2\mu+3)}\sigma_{\mu}, \qquad \mathcal{G}_{j,\mu}^{\mathrm{enr}}(\phi_i)=\frac{5\mu+7}{6(2\mu+3)}\sigma_{\mu}, \qquad i,j=1,2,3, \quad i\neq j.
\end{equation*}
Then, by~\eqref{rhoetau} and~\eqref{notimpsup}, the theorem is proved. 
\end{proof}

\section{Numerical experiments}\label{sec3}
In this section, we examine the efficacy of the proposed enrichment strategies using various examples. We analyze the performance across the following set of test functions
\begin{eqnarray*}
    &&f_1(x,y)=e^{x+y}, \qquad f_2(x,y)=\frac{1}{x^2+y^2+8},\\
&&  f_3(x,y)=\cos(x+y+1), \\ && f_4(x,y)=\frac{\sqrt{64-81((x-0.5)^2+(y-0.5)^2)}}{9}-0.5.
\end{eqnarray*}
We use four different Delaunay triangulations (see Figure~\ref{Fig:regulatri}) obtained through the Shewchuk's triangle program~\cite{Shewchuk:1996:TEA}. Numerical results are presented in Figures~\ref{fun1}-\ref{fun4}. In these figures, we analyze the error in $L^1$-norm produced by the standard Crouzeix--Raviart finite element in comparison to the enriched finite elements $GN_{\gamma}$ and $PN_{\mu}$, with $\gamma=\mu=2$.

In our results, we observe that the approximation achieved with the novel enriched finite elements exceeds the accuracy of the traditional Crouzeix--Raviart finite element. Notably, this enhancement becomes more pronounced with the increasing number of triangles in the triangulation.

 \begin{figure}
  \centering
   \includegraphics[width=0.24\textwidth]{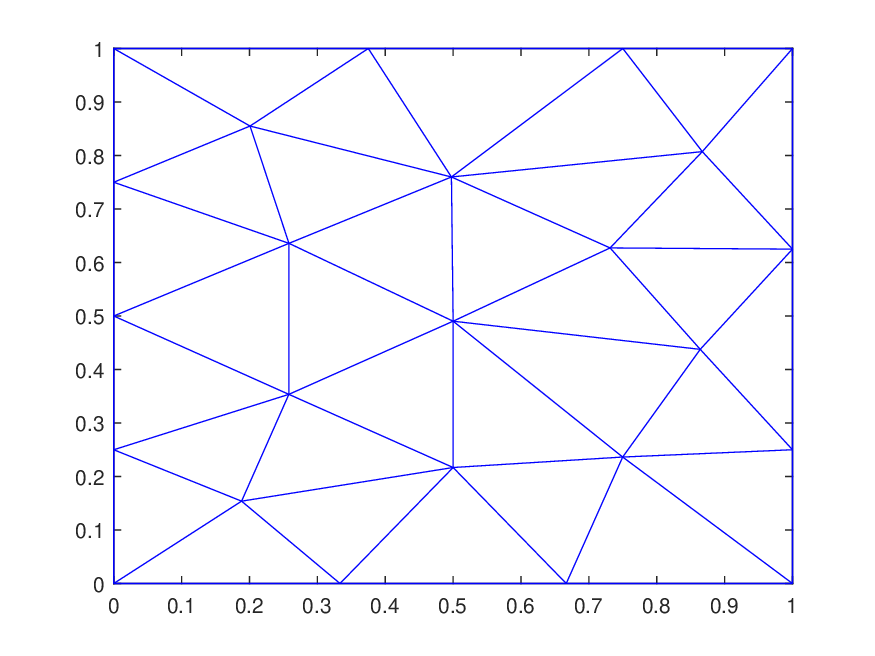} 
    \includegraphics[width=0.24\textwidth]{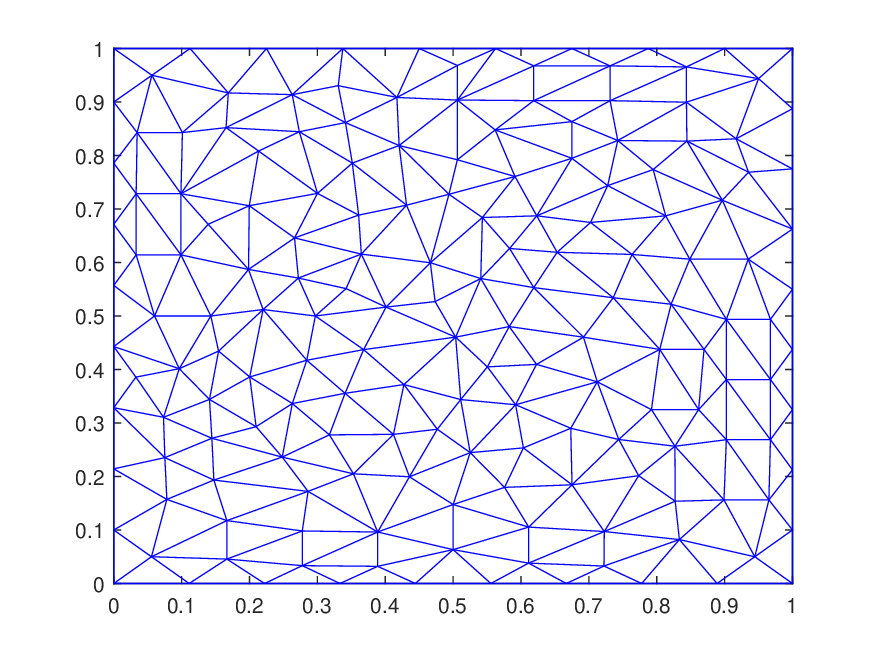} 
        \includegraphics[width=0.24\textwidth]{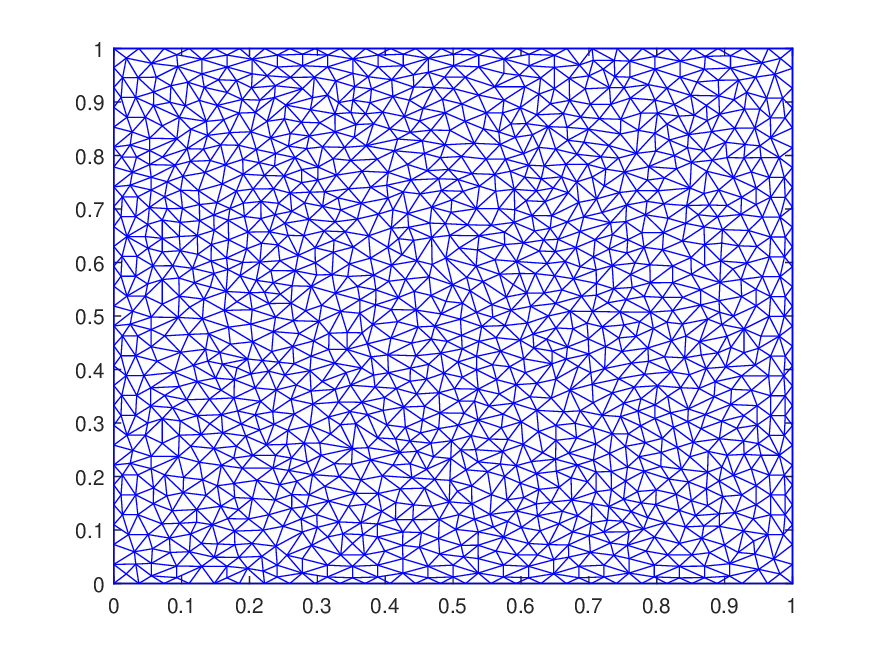}
            \includegraphics[width=0.24\textwidth]{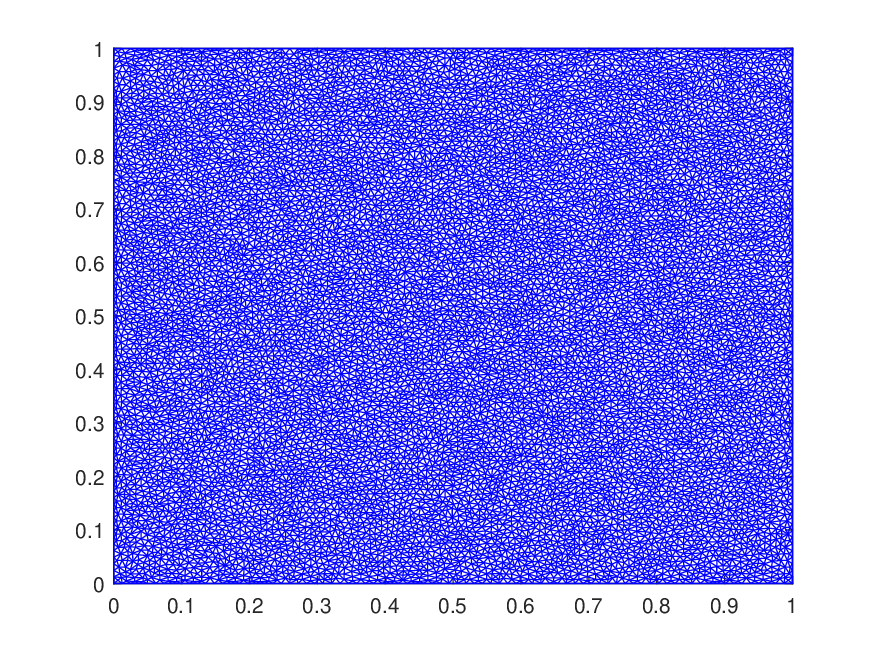}
         \caption{Delaunay triangulation of $N=33$, $N=306$, $N=2650$ and $N=23576$ tringles with no angle smaller than ${20}^{\circ}$.}
          \label{Fig:regulatri}
\end{figure}

 \begin{figure}
  \centering
   \includegraphics[width=0.49\textwidth]{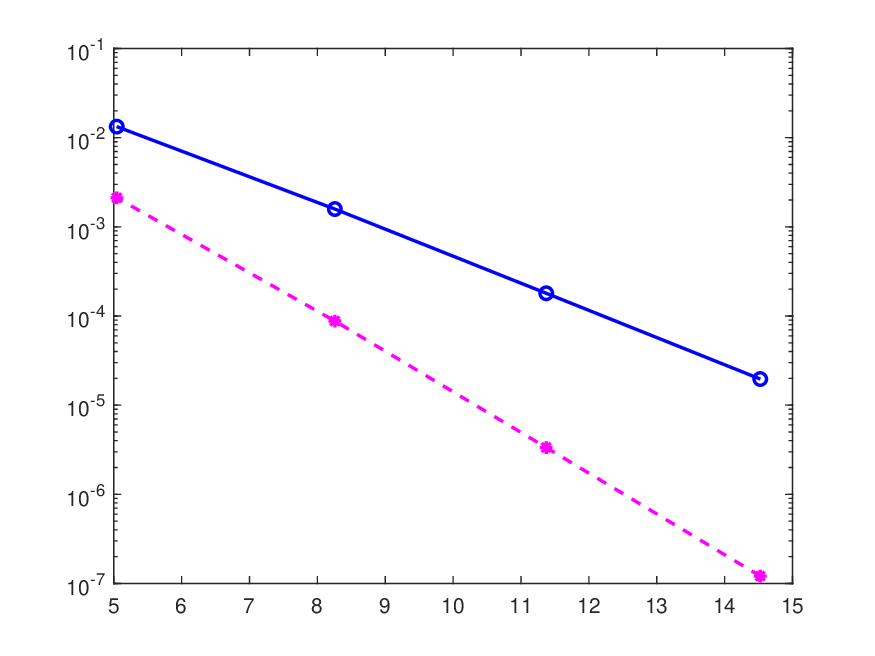} 
    \includegraphics[width=0.49\textwidth]{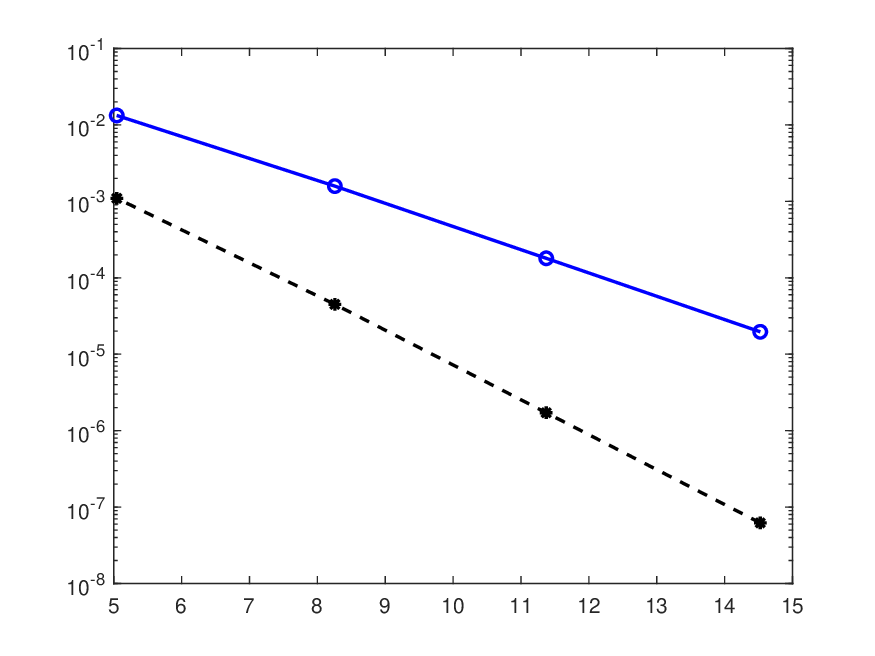} 
         \caption{Loglog plot of the errors in $L^1$-norm in approximating the function $f_1$. The blue line represents the trend of approximation errors obtained using the standard Crouzeix--Raviart finite element, the magenta line corresponds to the errors obtained with the enriched finite element $GN_{\gamma}$ with $\gamma=2$ (left), and the black line represents the trend of approximation errors using the finite element $PN_{\mu}$ with $\mu=2$ (right). The comparisons are conducted by using Delaunay triangulations of Figure~\ref{Fig:regulatri}.}
          \label{fun1}
\end{figure}

 \begin{figure}
  \centering
   \includegraphics[width=0.49\textwidth]{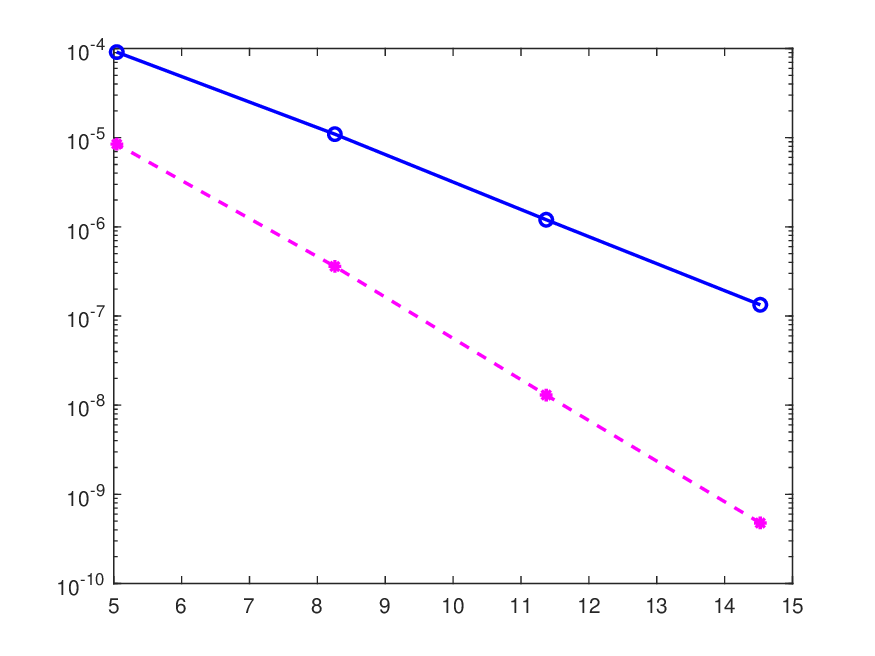} 
    \includegraphics[width=0.49\textwidth]{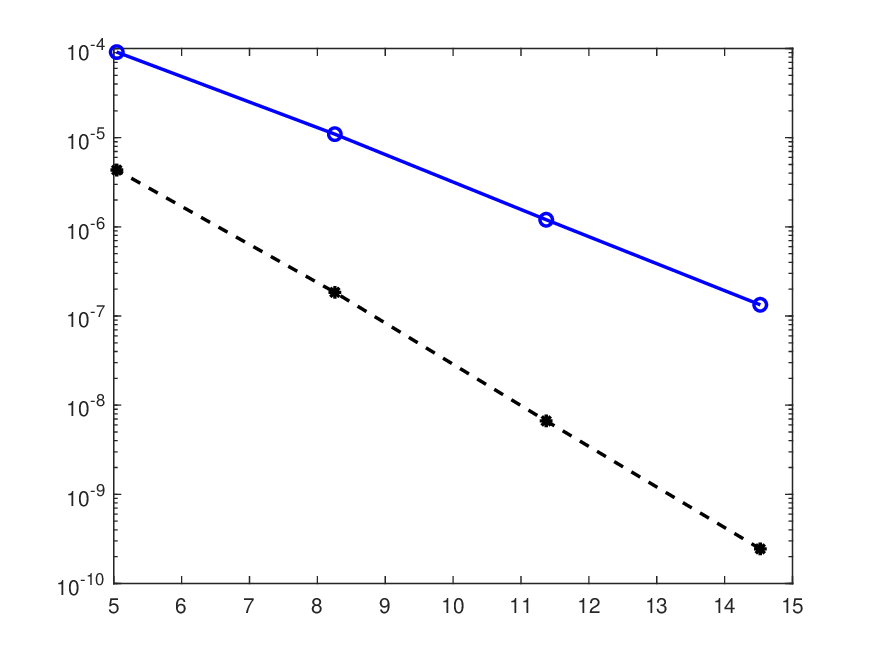} 
         \caption{Loglog plot of the errors in $L^1$-norm in approximating the function $f_2$. The blue line represents the trend of approximation errors obtained using the standard Crouzeix--Raviart finite element, the magenta line corresponds to the errors obtained with the enriched finite element $GN_{\gamma}$ with $\gamma=2$ (left), and the black line represents the trend of approximation errors using the finite element $PN_{\mu}$ with $\mu=2$ (right). The comparisons are conducted by using Delaunay triangulations of Figure~\ref{Fig:regulatri}.}
          \label{fun2}
\end{figure}

 \begin{figure}
  \centering
   \includegraphics[width=0.49\textwidth]{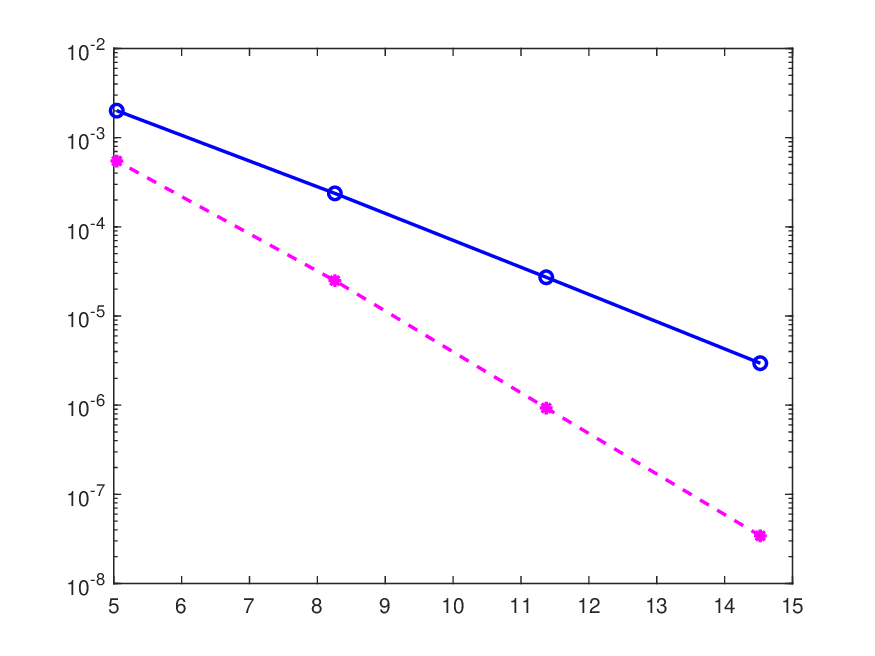} 
    \includegraphics[width=0.49\textwidth]{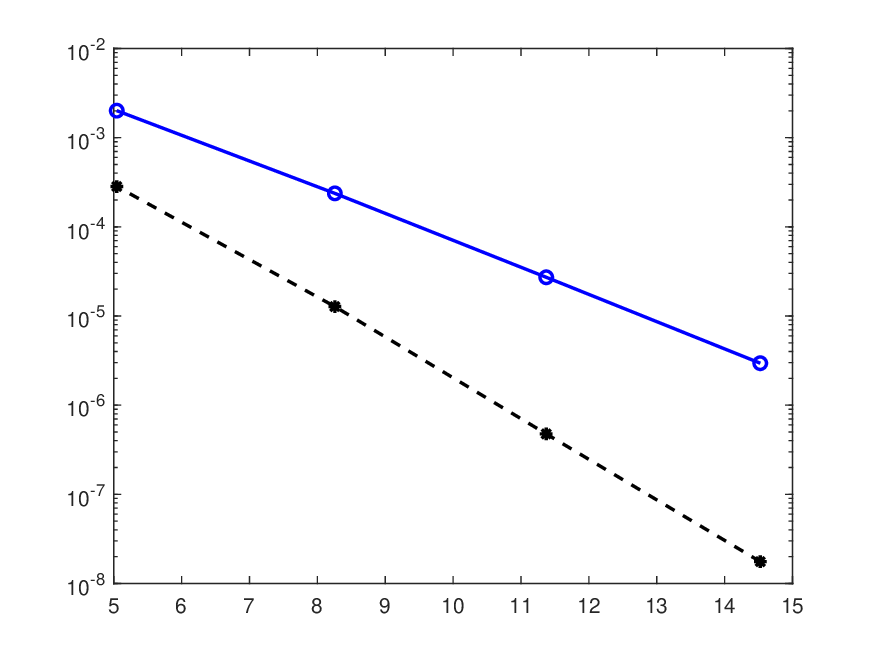} 
         \caption{Loglog plot of the errors in $L^1$-norm in approximating the function $f_3$. The blue line represents the trend of approximation errors obtained using the standard Crouzeix--Raviart finite element, the magenta line corresponds to the errors obtained with the enriched finite element $GN_{\gamma}$ with $\gamma=2$ (left), and the black line represents the trend of approximation errors using the finite element $PN_{\mu}$ with $\mu=2$ (right). The comparisons are conducted by using Delaunay triangulations of Figure~\ref{Fig:regulatri}.}
          \label{fun3}
\end{figure}

 \begin{figure}
  \centering
   \includegraphics[width=0.49\textwidth]{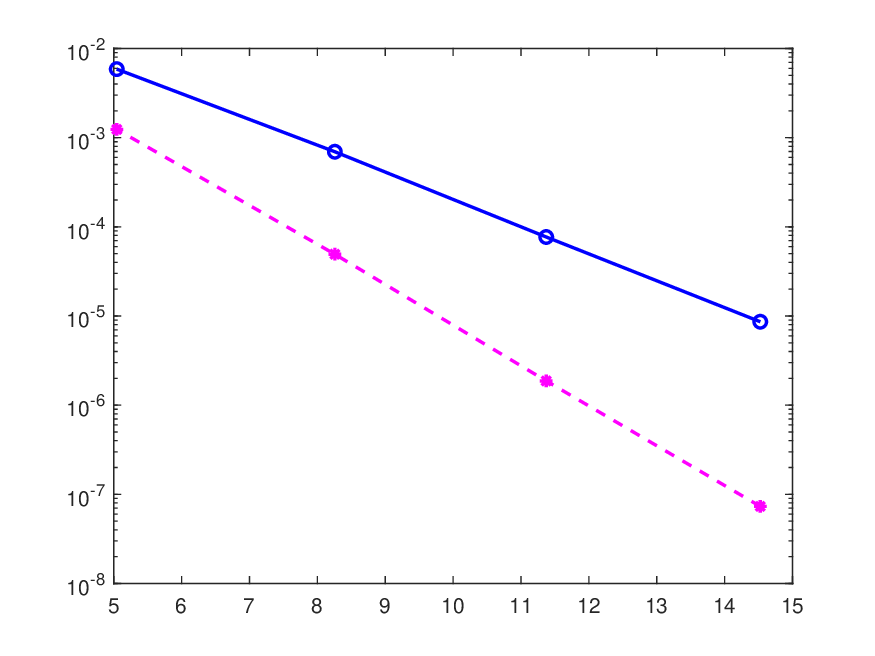} 
    \includegraphics[width=0.49\textwidth]{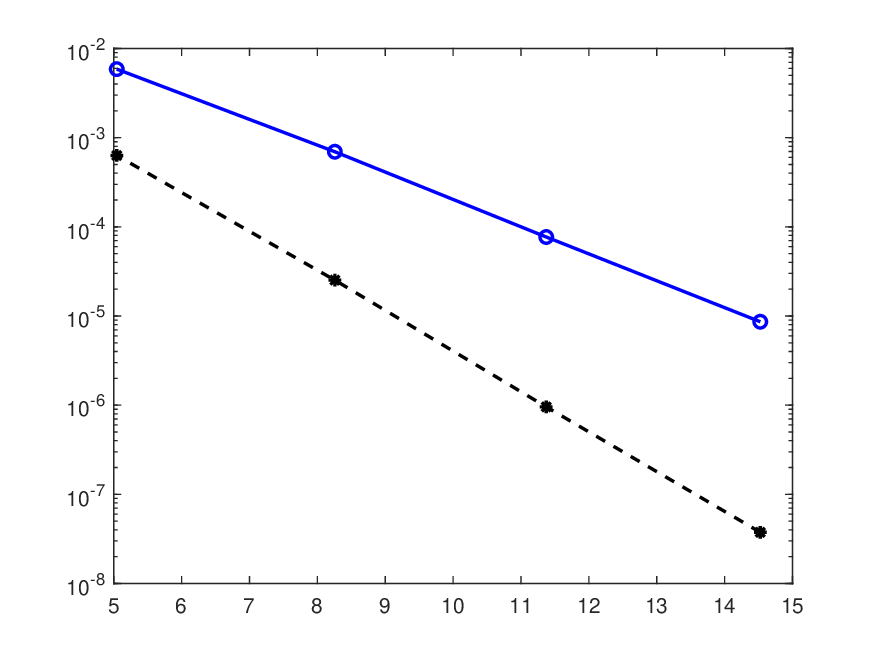} 
         \caption{Loglog plot of the errors in $L^1$-norm in approximating the function $f_4$. The blue line represents the trend of approximation errors obtained using the standard Crouzeix--Raviart finite element, the magenta line corresponds to the errors obtained with the enriched finite element $GN_{\gamma}$ with $\gamma=2$ (left), and the black line represents the trend of approximation errors using the finite element $PN_{\mu}$ with $\mu=2$ (right). The comparisons are conducted by using Delaunay triangulations of Figure~\ref{Fig:regulatri}.}
          \label{fun4}
\end{figure}

\section*{Acknowledgments}
This research has been achieved as part of RITA \textquotedblleft Research
 ITalian network on Approximation'' and as part of the UMI group ``Teoria dell'Approssimazione
 e Applicazioni''. The author is a member of the INdAM Research group GNCS.

\bibliographystyle{elsarticle-num}
\bibliography{bibliography}

\end{document}